\documentclass[a4paper,10pt]{article}

\usepackage[utf8]{inputenc}
\usepackage{amsmath}
\usepackage{amsbsy}
\usepackage{amssymb}
\usepackage{amscd,amsthm}
\usepackage{graphicx}
\usepackage[mathcal]{eucal}
\usepackage{verbatim}
\usepackage[dvigs]{epsfig}
\usepackage{dsfont}
\usepackage{longtable}
\usepackage{marvosym}
\usepackage{anysize}

\newcommand{\e}{\varepsilon}

\newcommand{\Z}{\mathds{Z}}

\newcommand{\N}{\mathds{N}}
\newcommand{\R}{\mathds{R}}
\newcommand{\Pz}{\mathds{P}}

\newcommand{\q}{\quad}

\newtheorem{thm}{Theorem}[section]
\newtheorem{lem}[thm]{Lemma}
\newtheorem{kor}[thm]{Corollary}
\newtheorem{conj}[thm]{Conjecture}
\newtheorem{prop}[thm]{Proposition}
\newtheorem*{bp}{Bertrand's Postulate}
\newtheorem{thmE}{Theorem}

\newcounter{tmp}

\theoremstyle{remark}
\newtheorem*{rema}{Remark}

\theoremstyle{definition}
\newtheorem*{defi}{Definition}

\title{On generalized Ramanujan primes}
\author{Christian Axler}

\begin{document}

\maketitle

\begin{abstract}
In this paper we establish several results concerning the generalized Ramanujan primes. For $n\in\N$ and $k \in \R_{> 1}$ we give estimates for the $n$th
$k$-Ramanujan prime which lead both to generalizations and to improvements of the results presently in the literature. Moreover, we obtain results about the
distribution of $k$-Ramanujan primes. In addition, we find explicit formulae for certain $n$th $k$-Ramanujan primes. As an application, we prove that a
conjecture of Mitra, Paul and Sarkar \cite{mps} concerning the number of primes in certain intervals holds for every sufficiently large positive integer.
\end{abstract}

\section{Introduction}

Ramanujan primes, named for the Indian mathematician Srinivasa Ramanujan, were introduced by Sondow \cite{so3} in 2005 and have their origin in Bertrand's
postulate.

\begin{bp}
For each $n \in \mathds{N}$ there is a prime number $p$ with $n < p \leq 2n$.
\end{bp}

\noindent
Bertrand's postulate was proved, for instance, by Tchebychev \cite{tch} and by Erdös \cite{er2}. In 1919, Ramanujan \cite{ram} proved an extension of
Bertrand's postulate by showing that
\begin{displaymath}
\pi(x) - \pi \left( \frac{x}{2} \right) \geq 1 \; (\text{respectively} \; 2,3,4,5,\ldots)
\end{displaymath}
for every
\begin{displaymath}
x \geq 2 \; ( \text{respectively} \; 11,17,29,41, \ldots).
\end{displaymath}
Motivated by the fact $\pi(x) - \pi(x/2) \rightarrow \infty$ as $x \rightarrow \infty$ by the Prime Number Theorem (PNT), Sondow \cite{so3} defined the
number $R_n \in \N$ for each $n \in \N$ as the smallest positive integer such that the inequality $\pi(x) - \pi(x/2) \geq n$ holds for every $x \geq R_n$.
He called the number $R_n$ the \emph{$n$th Ramanujan prime}, because $R_n \in \Pz$ for every $n\in \N$, where $\Pz$ denotes the set of prime numbers.

This can be generalized as follows. Let $k \in (1,\infty)$. Again, the PNT implies that $\pi(x) - \pi(x/k) \to \infty$ as $x \to \infty$ and Shevelev
\cite{sh} introduced the $n$th $k$-Ramanujan prime as follows.

\begin{defi}
Let $k>1$ be real. For every $n \in \N$, let
\begin{displaymath}
R_n^{(k)} = \min \{ m \in \N \mid \pi(x) - \pi(x/k) \geq n \;\, \text{for every real} \; x \geq m \}.
\end{displaymath}
This number is prime and it is called the \emph{$n$th $k$-Ramanujan prime}. Since $R_n^{(2)} = R_n$ for every $n \in \N$, the numbers $R_n^{(k)}$ are also
called \emph{generalized Ramanujan primes}.
\end{defi}

\begingroup
\setcounter{tmp}{\value{thmE}}
\setcounter{thmE}{0}
\renewcommand\thethmE{\Alph{thmE}}

\noindent
In 2009, Sondow \cite{so} showed that
\begin{equation} \label{101}
R_n \sim p_{2n} \q\q (n \to \infty),
\end{equation}
where $p_n$ denotes the $n$th prime number. Further, he proved that
\begin{equation} \label{102}
R_n > p_{2n}
\end{equation}
for every $n \geq 2$. In 2011, Amersi, Beckwith, Miller, Ronan and Sondow \cite{abmrs} generalized the asymptotic formula \eqref{101} to $k$-Ramanujan
primes by showing that
\begin{equation} \label{103}
R_n^{(k)} \sim p_{\lceil kn/(k-1) \rceil} \q\q (n \to \infty).
\end{equation}
In view of \eqref{103}, one may ask whether the inequality \eqref{102} can also be generalized to $k$-Ramanujan primes. We prove that this is indeed the
case. In fact, we derive further inequalities concerning the $n$th $k$-Ramanujan prime, by constructing explicit constants $n_0, n_1, n_2, n_3 \in \N$
depending on a series of parameters including $k$ (see \eqref{312}, \eqref{333}, Theorem \ref{t404}, Theorem \ref{t503}, respectively), such that the
following theorems hold.

\begin{thmE}[see Theorem 3.1] \label{t001}
Let $t \in \Z$ with $t > - \lceil k/(k-1) \rceil$. Then for every $n \geq n_0$,
\begin{equation} \label{104}
R_n^{(k)} > p_{\lceil kn/(k-1) \rceil + t}.
\end{equation}
\end{thmE}

\noindent
Another problem which arises is to find a minimal bound $m = m(k,t)$ such that the inequality \eqref{104} holds for all $n \geq m$. For the case $t=0$, we
introduce the following

\begin{defi}
For $k > 1$ let
\begin{equation} \label{Nk05}
N(k) = \min \{ m \in \N \mid R_n^{(k)} > p_{\lceil kn/(k-1) \rceil} \; \text{for every} \; n \geq m\}.
\end{equation}
\end{defi}

\noindent
In Section 3.2 we prove the following theorem giving an explicit formula for $N(k)$.

\begin{thmE}[see Corollary 3.11]
If $k \geq 745.8$, then
\begin{displaymath}
N(k) = \pi(3k) - 1.
\end{displaymath}
\end{thmE}

\noindent
Theorem \ref{t001} is supplemented by the following upper bound for $n$th $k$-Ramanujan prime.

\begin{thmE}[see Theorem 3.21] \label{thm056}
Let $\e_1\geq 0$, $\e_2 \geq 0$ and $\e_1+\e_2 \neq 0$. Then for every $n \geq n_1$,
\begin{displaymath}
R_n^{(k)} \leq (1+\e_1)p_{\lceil (1+\e_2)kn/(k-1) \rceil}.
\end{displaymath}
\end{thmE}

\noindent
By \cite{abmrs}, there exists a positive constant $\beta_1 = \beta_1(k)$ such that for every sufficiently large $n$,
\begin{equation} \label{106}
| R_n^{(k)} - p_{\lfloor kn/(k-1) \rfloor} | < \beta_1 n \log \log n.
\end{equation}
In Theorem \ref{t001}, we actually obtain a lower bound for $R_n^{(k)} - p_{\lceil kn/(k-1) \rceil}$ improving the lower bound given in \eqref{106}. The
next theorem yields an improvement of the upper bound.

\begin{thmE}[see Theorem 4.4]
There exists a positive constant $\gamma$, depending on a series of parameters including $k$, such that for every $n \geq n_2$,
\begin{displaymath}
R_n^{(k)} - p_{\lceil kn/(k-1) \rceil} < \gamma n.
\end{displaymath}
\end{thmE}

\noindent
Let $\pi_k(x)$ be the number of $k$-Ramanujan primes less than or equal to $x$. Using PNT, Amersi, Beckwith, Miller, Ronan and Sondow \cite{abmrs} proved
that there exists a positive constant $\beta_2 = \beta_2(k)$ such that for every sufficiently large $n$,
\begin{equation} \label{107}
\left | \frac{k-1}{k} - \frac{\pi_k(n)}{\pi(n)} \right| \leq \frac{\beta_2 \log \log n}{\log n}.
\end{equation}
In Section 4 we prove the following two theorems which lead to an improvement of the lower and upper bound in \eqref{107}.

\begin{thmE}[see Proposition 5.1]
If $x \geq R_{N(k)}^{(k)}$ with $N(k)$ defined above, then
\begin{displaymath}
\frac{\pi_k(x)}{\pi(x)} < \frac{k-1}{k}.
\end{displaymath}
\end{thmE}

\begin{thmE}[see Theorem 5.3]
There exists a positive constant $c$, depending on a series of parameters including $k$, such that for every $x \geq n_3$,
\begin{displaymath}
\frac{k-1}{k} - \frac{\pi_k(x)}{\pi(x)} < \frac{c}{\log x}.
\end{displaymath}
\end{thmE}

\noindent
In 2009, Mitra, Paul and Sarkar \cite{mps} stated a conjecture concerning the number of primes in certain intervals, namely that
\begin{displaymath}
\pi(mn) - \pi(n) \geq m-1
\end{displaymath}
for all $m,n \in \N$ with $n \geq \lceil 1.1\log(2.5m) \rceil$. In Section 5 we confirm the conjecture for large $m$.

\begin{thmE}[see Corollary 6.3]
If $m$ is sufficiently large and $n \geq \lceil 1.1\log(2.5m) \rceil$, then $\pi(mn) - \pi(n) \geq m-1$.
\end{thmE}

\endgroup

\section{Some simple properties of $k$-Ramanujan primes}

We begin with

\begin{prop} \label{p201}
The following three properties hold for $R_n^{(k)}$.
\begin{enumerate}
 \item[\emph{(i)}] Let $k_1, k_2 \in \R$ with $k_2 > k_1 > 1$. Then $R_n^{(k_1)} \geq R_n^{(k_2)}$.
 \item[\emph{(ii)}] $R_n^{(k)} \geq p_n$ for every $n \in \N$ and every $k > 1$.
 \item[\emph{(iii)}] For each $k$, the sequence $(R_n^{(k)})_{n \in \N}$ is strictly increasing.
\end{enumerate}
\end{prop}

\begin{proof}
The assertions follow directly from the definition of $R_n^{(k)}$.
\end{proof}

\begin{prop} \label{p202}
Let $k>1$ and let $n \in \N$ so that $R_n^{(k)} = p_n$. Then $R_m^{(k)} = p_m$ for every $m \leq n$.
\end{prop}

\begin{proof}
Let $n \in \N$ be such that $R_n^{(k)} = p_n$. By Proposition \ref{p201}(iii), we get $R_{n-1}^{(k)} < R_n^{(k)} = p_n$; i.e. $R_{n-1}^{(k)} \leq p_{n-1}$.
Using Proposition \ref{p201}(ii), we obtain $R_{n-1}^{(k)} = p_{n-1}$ and our proposition follows by induction.
\end{proof}

\begin{prop} \label{p204}
If $k \geq 2$, then
\begin{displaymath}
R_{\pi(k)}^{(k)} = p_{\pi(k)}.
\end{displaymath}
\end{prop}

\begin{proof}
By \cite[Satz 33, p.\ 60]{is}, we have $\pi(m)+\pi(n)<\pi(mn)$ for every $m,n \in \N$ with $m,n \geq 2$ and $\max \{m,n\} \ge 6$. Now it is easy to check
that the inequality
\begin{equation} \label{208}
\pi(m) + \pi(n) \leq \pi(mn)
\end{equation}
holds for every $m,n \in \N$. Let $t = \lfloor k \rfloor$ and $x \geq k$. Let $m \in \N$ be such that $mk \leq x < (m+1)k$. Using \eqref{208}, we get
\begin{displaymath}
\pi(x) - \pi \left( \frac{x}{k} \right) \geq \pi(mt) - \pi(m) \geq \pi(t) = \pi(k),
\end{displaymath}
i.e. $R_{\pi(k)}^{(k)} \leq k < p_{\pi(k)+1}$. Using Proposition \ref{p201}(ii), we obtain the required equality.
\end{proof}

\begin{kor} \label{k205}
If $k \geq 2$, then $R_n^{(k)} = p_n$ for every $n=1,\ldots, \pi(k)$.
\end{kor}

\begin{proof}
The claim follows from Proposition \ref{p202} and Proposition \ref{p204}.
\end{proof}

\noindent
For $1 < k < 2$ we can give more information on $R_n^{(k)}$.

\begin{prop} \label{p206}
We have:
\begin{enumerate}
 \item[\emph{(i)}] If $1 < k< 5/3$, then $R_n^{(k)} > p_n$ for every $n \in \N$.
 \item[\emph{(ii)}] If $5/3 \leq k < 2$, then $R_n^{(k)} = p_n$ if and only if $n=1$.
\end{enumerate}
\end{prop}

\begin{proof}
(i) If $1<k<3/2$, we set $x=2k$ and obtain $\pi(x) - \pi(x/k) = 0$, i.e. $R_1^{(k)} > 2k > p_1$. It remains to use Proposition \ref{p201}(iii). If $3/2\leq
k<5/3$, we set $x=3k$ and proceed as before.

(ii) Let $n=1$ and $5/3 \leq k < 2$. By Proposition \ref{p201}(ii) we get $p_1 \leq R_1^{(5/3)} = p_1$, i.e. $R_1^{(k)} = p_1$. Let $n \geq 2$. Then
$R_2^{(k)} \geq R_2^{(2)} = 11 > p_2$ and we use Proposition \ref{p201}(iii) as in the previous case.
\end{proof}

\noindent
The following property will be useful in Section 4.

\begin{prop} \label{p207}
For every $n$ and $k$,
\begin{displaymath}
\pi(R_n^{(k)}) - \pi \left( \frac{R_n^{(k)}}{k} \right) = n.
\end{displaymath}
\end{prop}

\begin{proof}
This easily follows from the definition of $R_n^{(k)}$.
\end{proof}

\noindent
Finally, we formulate an interesting property of the $k$-Ramanujan primes.

\begin{prop} \label{p208}
If $p \in \Pz \setminus \{2\}$, then for every $n \in \N$
\begin{displaymath}
R_n^{(k)} \neq kp-1.
\end{displaymath}
\end{prop}

\begin{proof}
It suffices to consider the case $kp \in \N$. Assume $R_n^{(k)} = kp-1$ for some $n \in \N$. Since $kp-1 > 2$, we obtain $kp \not\in \mathds{P}$. Let $r \in
\R$ with $0 \leq r < 1$. Using Proposition \ref{p207}, we get
\begin{displaymath}
\pi(kp+r) - \pi \left( \frac{kp+r}{k} \right) = \pi(R_n^{(k)}) - \pi \left( \frac{R_n^{(k)}}{k} \right) - 1 = n-1,
\end{displaymath}
which contradicts the definition of $R_n^{(k)}$.
\end{proof}

\section{Estimates for the $n$th $k$-Ramanujan prime}

From here on, we use the following notation. Let $m_1 \in \N$ and $s, a_1, \ldots, a_{m_1} \in \R$ with $s \geq 0$. We define
\begin{displaymath}
A(x) = \sum_{j=1}^{m_1} \frac{a_j}{\log^j x}
\end{displaymath}
and $Y_s = Y_s(a_1, \ldots, a_{m_1})$ so that
\begin{equation} \label{309}
\pi(x) > \frac{x}{\log x - 1 - A(x)} + s
\end{equation}
for every $x \geq Y_s$. Further, for $m_2 \in \N$ and $b_1, \ldots, b_{m_2} \in \R_{\geq 0}$ we define
\begin{displaymath}
B(x) = \sum_{j=1}^{m_2}\frac{b_j}{\log^j x}
\end{displaymath}
and $X_0 = X_0(b_1, \ldots, b_{m_2})$ so that
\begin{equation} \label{310}
\pi(x) < \frac{x}{\log x - 1 - B(x)}
\end{equation}
for every $x \geq X_0$. In addition, let $X_1 = X_1(k, a_1, \ldots, a_{m_1}, b_1, \ldots, b_{m_2})$ be such that
\begin{equation} \label{311}
\log k - B(kx) + A(x) \geq 0
\end{equation}
for every $x \geq X_1$.

\begin{rema}
From \cite{pan3} it follows directly that there exist parameters $s,m_1, m_2, a_1, \ldots, a_{m_{1}}, b_1, \ldots, b_{m_{2}}, Y_s$ and $X_0$ such that the
inequalities \eqref{309} and \eqref{310} are fulfilled.
\end{rema}

\begin{rema}
It is clear that $B(x) > A(x)$ for every $x \geq \max\{Y_s,X_0\}$. Hence $b_1 > a_1$.
\end{rema}

\subsection{On a lower bound for the $n$th $k$-Ramanujan prime}

\subsubsection{A lower bound for the $n$th $k$-Ramanujan prime}

The theorem below implies that the inequality \eqref{102} can be generalized, in view of \eqref{103}, to $k$-Ramanujan primes. 

\begin{thm} \label{t301}
Let $t \in \Z$ with $t > - \lceil k/(k-1) \rceil$. Then
\begin{displaymath}
R_n^{(k)} > p_{\lceil kn/(k-1) \rceil + t}
\end{displaymath}
for every $n \in \N$ with
\begin{equation} \label{312}
n \geq n_0 = \frac{k-1}{k} ( \pi(X_2) - t + 1),
\end{equation}
where $r \geq (t+1)(k-1)/k$ and
\begin{equation} \label{X2}
X_2 = X_2(k, r, m_1, m_2, a_1, \ldots, a_{m_{1}}, b_1, \ldots, b_{m_{2}}) = \max \{ X_0, kX_1, kY_r \}.
\end{equation}
\end{thm}

\begin{proof}
Let $x \geq X_2/k$. Then the inequality \eqref{311} is equivalent to
\begin{displaymath}
\frac{x}{\log x - 1 - A(x)} \geq \frac{x}{\log(kx) - 1 - B(kx)}
\end{displaymath}
and we get
\begin{equation}  \label{314}
\pi(x) > \frac{x}{\log x - 1 -A(x)} + r \geq \frac{x}{\log kx - 1 - B(kx)} + r > \frac{\pi(kx)}{k} + r.
\end{equation}
By setting $x=p_{\lceil kn/(k-1) \rceil + t}/k \geq X_2/k$ in \eqref{314}, we obtain
\begin{displaymath}
\pi \left( \frac{1}{k} \, p_{\lceil kn/(k-1) \rceil + t} \right) > \frac{1}{k} \left( \left\lceil \frac{kn}{k-1} \right\rceil + t \right) + r.
\end{displaymath}
Hence,
\begin{displaymath}
\pi(p_{\lceil kn/(k-1)\rceil + t}) - \pi \left( \frac{1}{k} \, p_{\lceil kn/(k-1)\rceil + t} \right) < \frac{k-1}{k} \left\lceil \frac{kn}{k-1} \right\rceil
+ t - \frac{t}{k} - r \leq n,
\end{displaymath}
and we apply the definition of $R_n^{(k)}$.
\end{proof}

\begin{kor} \label{k302}
We have
\begin{displaymath}
\liminf_{n \to \infty} (R_n^{(k)} - p_{\lceil kn/(k-1) \rceil})= \infty.
\end{displaymath}
\end{kor}

\begin{proof}
From Theorem \ref{t301}, it follows that for every $t \in \N$ there is an $N_0 \in \N$ such that for every $n \geq N_0$,
\begin{displaymath}
R_n^{(k)} - p_{\lceil kn/(k-1) \rceil} \geq p_{\lceil kn/(k-1) \rceil + t} - p_{\lceil kn/(k-1) \rceil} \geq 2t.
\end{displaymath}
This proves our corollary.
\end{proof}

\begin{rema}
In 2013, Sondow \cite{so2} raised the question whether the sequence $(R_n - p_{2n})_{n \in \N}$ is unbounded. Corollary \ref{k302} implies that this is
indeed the case.
\end{rema}

\begin{kor} \label{k303}
If $n \geq \max \{ 2, (k-1)\pi(X_2)/k \}$, where $X_2$ is defined by \eqref{X2}, then
\begin{displaymath}
R_n^{(k)} - p_{\lceil kn/(k-1) \rceil} \geq 6.
\end{displaymath}
\end{kor}

\begin{proof}
We set $t=1$ in Theorem \ref{t301}. Then for every $n \geq (k-1)\pi(X_2)/k$ we obtain
\begin{displaymath}
R_n^{(k)} \geq p_{\lceil kn/(k-1) \rceil + 2} \geq p_{\lceil kn/(k-1) \rceil + 1} + 2 \geq p_{\lceil kn/(k-1) \rceil} + 4.
\end{displaymath}
Since there is no prime triple of the form $(p,p+2,p+4)$ for $p>3$, we are done.
\end{proof}

\noindent
Now we find an explicit value for $X_2$ in the case $t=0$.

% \begin{lem} \label{l304}
% If $x \geq 470077$, then
% \begin{displaymath}
% \pi(x) > \frac{x}{\log x - 1 - \frac{1}{\log x}} + 1.
% \end{displaymath}
% \end{lem}
% 
% \begin{proof}
% We consider the function $g(x) = 2.65x-\log^4x$. Then $g(x) > 0$ for every $x \geq e^7$ and we get
% \begin{displaymath}
% \frac{x}{\log x - 1 - \frac{1}{\log x} - \frac{2.65}{\log^2 x}} > \frac{x}{\log x - 1 - \frac{1}{\log x}} + 1
% \end{displaymath}
% for every $x \geq e^7$. By Corollary 3.11 of \cite{ca2}, the inequality $\pi(x) > x/(\log x - 1 - 1/\log x) + 1$ holds for every $x \geq 38168363$. We set
% $h(x) = x/(\log x - 1 - 1/\log x)$. Then $h'(x) \geq 0$ for every $x \geq 12.8$ and we check with a computer that $\pi(p_i) \geq h(p_{i+1}) + 1$ for every
% $\pi(470077) \leq i \leq \pi(38168363)$.
% \end{proof}

\begin{prop} \label{p305}
Let $X_3 = X_3(k) = \max \{470077k, kr(k)\}$, where
\begin{displaymath}
r(k) = \frac{1}{k} \, \exp \left( \sqrt{ \max \left\{ \frac{3.83}{\log k} -1, 0 \right\} } \right).
\end{displaymath}
Then
\begin{displaymath}
R_n^{(k)} > p_{\lceil kn/(k-1) \rceil}
\end{displaymath}
for every
\begin{displaymath}
n \geq \frac{k-1}{k} ( \pi(X_3) + 1).
\end{displaymath}
\end{prop}

\begin{proof}
By Corollary 3.6 of \cite{ca2}, it is easy to see that the inequality
\begin{equation} \label{315}
\pi(x) > \frac{x}{\log x - 1 - \frac{1}{\log x}} + 1
\end{equation}
holds for every $x \geq 38168363$. A computer check shows that \eqref{315} also holds for every $470077 \leq x \leq 38168363$. We choose $t=0$ in Theorem
\ref{t301}. Then $r = (k-1)/k$. We set $A(x) = 1/\log x$ and $Y_r = 470077$. By \eqref{315}, we get that the inequality \eqref{309} holds for every $x \geq
Y_r$. By choosing $b_1 = 1$, $b_2 = 3.83$ and $X_0 = 9.25$, we can use the third inequality in Corollary 3.5 of \cite{ca2}. Let $x \geq r(k)$. Then it is
easy to show that the inequality \eqref{311} holds. Now our proposition follows from Theorem \ref{t301}.
\end{proof}

\noindent
For Ramanujan primes we yield the following result.

\begin{kor} \label{k306}
If $n \geq 4$, then
\begin{displaymath}
R_n - p_{2n} \geq 6.
\end{displaymath}
\end{kor}

\begin{proof}
We set $t = 1$ and $k=2$ in Theorem \ref{t301}. Then $r=1$. From Corollary \ref{k303} and from the proof of Proposition \ref{p305}, if follows that $R_n -
p_{2n} \geq 6$ for all $n \geq \pi(X_3(2))/2 = \pi(\max \{940154, 2r(2)\})/2 = 37098$. We check with a computer that the inequality $R_n - p_{2n} \geq 6$
also holds for every $4 \leq n \leq 37097$.
\end{proof}

\begin{rema}
Since $R_2 - p_{4} = 4$ and $R_3 - p_{6} = 4$, Corollary \ref{k306} gives a positive answer to the question raised by Sondow \cite{so2}, whether $\min \{
R_n - p_{2n} \mid n \geq 2 \} = 4$.
\end{rema}

\subsubsection{An explicit formula for $N(k)$}

In the introduction we defined $N(k)$ to be the smallest positive integer so that
\begin{displaymath}
R_n^{(k)} > p_{\lceil kn/(k-1) \rceil}
\end{displaymath}
for every $n \geq N(k)$. By Proposition \ref{p305}, we get
\begin{displaymath} 
N(k) \leq \left \lceil \frac{k-1}{k} (\pi( \max \{470077k, kr(k)\}) + 1) \right \rceil
\end{displaymath}
for every $k > 1$. We can significantly improve this inequality in the following case.

\begin{thm} \label{t307}
If $k \geq 745.8$, then
\begin{displaymath}
N(k) \leq \pi(3k)-1.
\end{displaymath}
\end{thm}

\begin{proof}
Since $\pi(3k) - \pi(3k/k) = \pi(3k) - 2 < \pi(3k) - 1$, we have
\begin{equation} \label{316}
R_{\pi(3k)-1}^{(k)} > 3k \geq p_{\pi(3k)}
\end{equation}
By induction and Proposition \ref{p201}, we obtain that
\begin{equation} \label{317}
R_n^{(k)} > p_{n+1}
\end{equation}
for every $n \geq \pi(3k)-1$. We set $A(x) = - 7.1/\log x$, $s=1$ and $Y_1 = 3$. Then, as in the proof of \eqref{315}, we obtain that the inequality
\eqref{309} holds for every $x \geq Y_1$. By setting $B(x)=1.17/\log x$ and $X_0 = 5.43$ and using Corollary 3.5 of \cite{ca2}, we see that the inequality
\eqref{310} holds. Let
\begin{displaymath}
\widetilde r(k) = \exp \left( \sqrt{7.1 + \frac{1}{4}\left( \log k - \frac{8.27}{\log k} \right)^2} - \frac{1}{2} \left( \log k - \frac{8.27}{\log k}
\right) \right).
\end{displaymath}
It is easy to see that $x \geq \widetilde r(k)$ implies the inequality
\eqref{311}. By Theorem \ref{t301}, we obtain
\begin{equation}  \label{318}
N(k) \leq \left \lceil \frac{k-1}{k} ( \pi(X_4) + 1) \right \rceil,
\end{equation}
where $X_4 = X_4(k) = \max\{ 5.43, 3k, k\widetilde r(k)\}$. Since $\widetilde r(k)$ is decreasing, from $\widetilde r(745.8) \leq 2.999966$ we get that
$\widetilde r(k) \leq 3$ and therefore $X_4 = 3k$ for every $k \geq 745.8$. Since $\pi(3k)+1 \leq k$ for every $k \geq 745.8$, we obtain $N(k) \leq
\pi(3k) + 1$ by \eqref{318}. Finally, we apply \eqref{317}.
\end{proof}

\begin{rema}
Similarly to the proof of \eqref{316}, we obtain in general that for every real $r \geq 2/k$,
\begin{displaymath}
R_{\pi(rk) - \pi(r) + 1}^{(k)} > p_{\pi(rk)}.
\end{displaymath}
\end{rema}

\noindent
Next, we find a lower bound for $N(k)$.

\begin{prop} \label{p308}
For every $k> 1$,
\begin{displaymath}
N(k) > \pi(k).
\end{displaymath}
\end{prop}

\begin{proof}
First, let $k \geq 2$. Using Proposition \ref{p204}, we get
\begin{equation} \label{319}
R_{\pi(k)}^{(k)} < p_{\lceil k\pi(k)/(k-1) \rceil}.
\end{equation}
Hence, $N(k) > \pi(k)$ for every $k \geq 2$. The asserted inequality clearly holds for every $1<k<2$.
\end{proof}

\noindent
In order to prove a sharper lower bound for $N(k)$, see Theorem \ref{t311}, we need the following lemma.

\begin{lem} \label{l309}
Let $r,s \in \R$ with $r > s > 0$. If $t \geq s/r \cdot R_{\pi(r)}^{(r/s)}$, then
\begin{displaymath}
\pi(r) + \pi(t) \leq \pi \left( \frac{rt}{s} \right).
\end{displaymath}
\end{lem}

\begin{proof}
Since $rt/s \geq R_{\pi(r)}^{(r/s)}$, the claim follows from the definition of $R_n^{(k)}$.
\end{proof}

\begin{prop} \label{p310}
If $m,n \in \N$ with $m,n \geq 5$ and $\max \{m,n\} \geq 18$, then
\begin{displaymath}
\pi(m) + \pi(n) \leq \pi \left( \frac{mn}{3} \right).
\end{displaymath}
\end{prop}

\begin{proof}
Without loss of generality, let $m \geq n$. First, we consider the case $m \geq n \geq 20$. By \cite[p.\ 60]{tr}, we have $\pi(x) < 8x/(5\log x)$ for every
$x > 1$. Using the estimate $\pi(x) > x/\log x$ from \cite{rs}, we get
\begin{displaymath}
\pi \left( \frac{mn}{3} \right) - \pi(m) \geq \frac{mn}{3\log(mn/3)} - \frac{8m}{5\log m} \geq \frac{20m}{6\log m} - \frac{8m}{5\log m} =
\frac{52m}{30\log m} \geq \frac{8n}{5\log n} > \pi(n).
\end{displaymath}
So the proposition is proved, when $m \geq n \geq 20$. To obtain the required inequality for every $m \geq 18$ and $\min \{ m, 20 \} \geq n \geq 5$, we
consider the following table:
% Now let $m \geq 18$ and $\min \{ m, 20 \} \geq n \geq 5$. We have:
\begin{center}
\begin{tabular}{|c||c|c|c|c|c|c|c|c|c|c|c|c|c|c|c|c|}
\hline
$r$                                                        &5&6&7&8&9&10&11&12&13&14&15&16&17&18&19&20\\ \hline
$\max \{ r, \lceil 3/r \cdot R_{\pi(r)}^{(r/3)} \rceil \}$ &18&9&9&8&9&10&11&12&13&14&15&16&17&18&19&20 \\ \hline
\end{tabular}.
\end{center}
We apply Lemma \ref{l309} with $s=3$, $r=n$ and $t=m$.
\end{proof}

\begin{thm} \label{t311}
For every $k > 1$,
\begin{displaymath}
N(k) \geq \pi(3k)-1.
\end{displaymath}
\end{thm}

\begin{proof}
For every $1 < k < 5/3$ the claim is obviously true. For every $5/3 \leq k < 7/3$, we have $\pi(3k)-2=1$. Hence,
\begin{displaymath}
R_{\pi(3k)-2}^{(k)} \leq R_1^{(5/3)} = p_1 < p_{\lceil k(\pi(3k)-2)/(k-1) \rceil};
\end{displaymath}
i.e., $N(k) > \pi(3k)-2$. Similarly, for every $p_i/3 \leq k < p_{i+1}/3$, where $i=4,\ldots,8$, we check that
\begin{displaymath}
R_{\pi(3k)-2}^{(k)} \leq p_{\lceil k(\pi(3k)-2)/(k-1) \rceil}.
\end{displaymath}
Hence our theorem is proved for every $1 < k < 19/3$. Now, let $k \geq 19/3$. For $p_{\pi(3k)-1} \leq x < 3k$ and for $3k \leq x < 5k$ it is easy to see
that $\pi(x) - \pi(x/k) \geq \pi(3k) - 2$. So let $x \geq 5k$ and let $m\in \N$ be such that $m \geq 5$ and $mk \leq x < (m+1)k$. Since $\lfloor 3k \rfloor
\geq 19$, we use Proposition \ref{p310} to get the inequality
\begin{displaymath}
\pi(x) - \pi \left( \frac{x}{k} \right) \geq \pi(mk) - \pi(m) \geq \pi \left( \frac{\lfloor 3k \rfloor m}{3} \right) - \pi(m) \geq \pi(\lfloor 3k \rfloor)
= \pi(3k).
\end{displaymath}
Hence, altogether we have
\begin{displaymath}
\pi(x) - \pi \left( \frac{x}{k} \right) \geq \pi(3k)-2
\end{displaymath}
for every $x \geq p_{\pi(3k)-1}$ and it follows
\begin{equation} \label{320}
R_{\pi(3k)-2}^{(k)} \leq p_{\pi(3k)-1} \leq p_{\lceil k(\pi(3k)-2)/(k-1) \rceil}.
\end{equation}
Therefore $N(k) > \pi(3k)-2$.
\end{proof}

\begin{rema}
The proof of Theorem \ref{t311} yields $R_{\pi(3k)}^{(k)} \leq p_{\pi(5k)}$ for every $k \geq 19/3$.
\end{rema}

\noindent
From Theorem \ref{t307} and Theorem \ref{t311}, we obtain the following explicit formula for $N(k)$.

\begin{kor} \label{k312}
If $k \geq 745.8$, then
\begin{displaymath}
N(k) = \pi(3k) - 1.
\end{displaymath}
\end{kor}

\subsubsection{An explicit formula for $N_0(k)$}

By replacing ``$>$'' with ``$\geq$'' in the definition \eqref{Nk05} of $N(k)$, we get the following

\begin{defi}
For $k > 1$, let
\begin{displaymath}
N_0(k) = \min \{ m \in \N \mid R_n^{(k)} \geq p_{\lceil kn/(k-1) \rceil} \; \text{for every} \; n \geq m \}.
\end{displaymath}
\end{defi}

\noindent
Since $N_0(k) >\pi(k)$ for every $1<k<2$, it follows from \eqref{319} that
\begin{displaymath}
N_0(k) > \pi(k)
\end{displaymath}
is fulfilled for every $k > 1$. In the following case we obtain a sharper lower bound for $N_0(k)$.

\begin{thm} \label{t313}
If $k \geq 11/3$, then
\begin{displaymath}
N_0(k) \geq \pi(2k).
\end{displaymath}
\end{thm}

\begin{proof}
First, we show that
\begin{equation} \label{321}
\pi(x) - \pi \left( \frac{x}{k} \right) \geq \pi(2k)-1
\end{equation}
for every $x \geq p_{\pi(2k)-1}$. For $p_{\pi(2k)-1} \leq x < 2k$ and for $2k \leq x < 3k$, the inequality \eqref{321} is obviously true. Let $3k \leq x <
5k$. Since $\pi(3t) - \pi(2t) \geq 1$ for every $t \geq 11/3 = 1/3 \cdot R_1^{(3/2)}$, it follows $\pi(x) - \pi(x/k) \geq \pi(3k) - 2 \geq \pi(2k)-1$. So
let $x \geq 5k$ and let $l \in \N$ be such that $l \geq 5$ and $lk \leq x < (l+1)k$. Similarly to the proof of Proposition \ref{p310}, we get that
\begin{equation} \label{322}
\pi(m) + \pi(n) \leq \pi \left( \frac{mn}{2} \right)
\end{equation}
for every $m,n \geq 4$ with $\max \{ m,n \} \geq 6$. Since $k \geq 11/3$ we get $\lfloor 2k \rfloor > 7$, and, using \eqref{322}, we obtain the inequality
\begin{equation} \label{323}
\pi(x) - \pi \left( \frac{x}{k} \right) \geq \pi(lk) - \pi(l) \geq \pi \left( \frac{ \lfloor 2k \rfloor l}{2} \right) - \pi(l) \geq \pi(\lfloor 2k \rfloor)
= \pi(2k)
\end{equation}
for every $x \geq 5k$. Hence, we proved that the inequality \eqref{321} holds for every $x \geq p_{\pi(2k)-1}$. So, by the definition of $R_n^{(k)}$,
\begin{equation} \label{324}
R_{\pi(2k)-1}^{(k)} \leq p_{\pi(2k)-1} < p_{\lceil k(\pi(2k)-1)/(k-1) \rceil},
\end{equation}
which gives the required inequality.
\end{proof}

\noindent
Using \eqref{324}, we get an improvement of Corollary \ref{k205}.

\begin{kor} \label{k314}
If $k \geq 11/3$, then $R_n^{(k)} = p_n$ for every $1 \leq n \leq \pi(2k)-1$.
\end{kor}

\begin{proof}
Follows from Proposition \ref{p201}(ii), the left inequality in \eqref{324} and Proposition \ref{p202}.
\end{proof}

\noindent
To prove an upper bound for $N_0(k)$, the following proposition will be useful.

\begin{prop} \label{p315}
If $k \geq 29/3$, then
\begin{displaymath}
R_{\pi(2k)}^{(k)} = p_{\pi(2k)+1}.
\end{displaymath}
\end{prop}

\begin{proof}
Since $\pi(2k) - \pi(2k/k) < \pi(2k)$ and $p_{\pi(2k)} \leq 2k < p_{\pi(2k)+1}$, we have $R_{\pi(2k)}^{(k)} \geq p_{\pi(2k)+1}$ for every $k>1$. To prove 
$R_{\pi(2k)}^{(k)} \leq p_{\pi(2k)+1}$, it suffices to show that
\begin{equation} \label{325}
\pi(x) - \pi \left( \frac{x}{k} \right) \geq \pi(2k)
\end{equation}
for every $x \geq p_{\pi(2k)+1}$. It is clear that \eqref{325} is true for every $p_{\pi(2k)+1} \leq x < 3k$. Let $3k \leq x < 5k$. We have $\pi(3t) -
\pi(2t) \geq 2$ for every $t \geq 29/3 = 1/3 \cdot R_2^{(3/2)}$ and thus \eqref{325} holds. By \eqref{323} we already have that the inequality \eqref{325}
also holds for every $x \geq 5k$.
\end{proof}

\noindent
We can show more than in Corollary \ref{k314} for the following case.

\begin{kor} \label{k316}
If $k \geq 29/3$, then:
\begin{enumerate}
 \item[\emph{(i)}] $R_n^{(k)} = p_n$ if and only if $1 \leq n \leq \pi(2k)-1$.
 \item[\emph{(ii)}] $R_n^{(k)} = p_{n+1}$ if and only if $\pi(2k) \leq n \leq \pi(3k)-2$.
\end{enumerate}
\end{kor}

\begin{proof}
(i) From Proposition \ref{p315}, we get $R_n^{(k)} > p_n$ for every $n \geq \pi(2k)$. So, if $R_n^{(k)} = p_n$ then $1 \leq n \leq \pi(2k)-1$. On the other
hand, if $1 \leq n \leq \pi(2k)-1$, we use Corollary \ref{k314} to obtain $R_n^{(k)} = p_n$.

(ii) Let $\pi(2k) \leq n \leq \pi(3k)-2$. By (i), we get
\begin{equation} \label{326}
R_n^{(k)} \geq p_{n+1}.
\end{equation}
From \eqref{320} we have $R_{\pi(3k)-2}^{(k)} \leq p_{\pi(3k)-1}$. Using \eqref{326}, we obtain
\begin{equation} \label{327}
R_{\pi(3k)-2}^{(k)} = p_{\pi(3k)-1}.
\end{equation}
Now we prove by induction, that
\begin{displaymath}
R_{\pi(3k)-2-j}^{(k)} = p_{\pi(3k)-1-j}
\end{displaymath}
for every $j =0,\ldots,\pi(3k)-\pi(2k)-2$. For $j=0$, see \eqref{327}. Let $j \in \{ 0,\ldots,\pi(3k)-\pi(2k)- 3 \}$. By \eqref{326}, we have
$R_{\pi(3k)-2-(j+1)}^{(k)} \geq p_{\pi(3k)-1-(j+1)}$. Using Proposition \ref{p201}, we obtain
\begin{equation} \label{328}
p_{\pi(3k)-1-(j+1)} \leq R_{\pi(3k)-2-(j+1)}^{(k)} < R_{\pi(3k)-2-j}^{(k)}.
\end{equation}
By the induction hypothesis, we know $R_{\pi(3k)-2-j}^{(k)} = p_{\pi(3k)-1-j}$. Now we use \eqref{328} to get $R_{\pi(3k)-2-(j+1)}^{(k)} =
p_{\pi(3k)-1-(j+1)}$. So, we proved by induction that $R_n^{(k)} = p_{n+1}$ for every $\pi(2k) \leq n \leq \pi(3k)-2$.

Now, let $n \in \N$ be such that $R_n^{(k)} = p_{n+1}$.
% We have to show that $\pi(2k) \leq n \leq \pi(3k)-2$.
From (i) and \eqref{317}, we obtain $\pi(2k) \leq n \leq \pi(3k)-2$.
% we have $n \geq \pi(2k)$. By
% \eqref{grnplus1}, we obtain $R_{\pi(3k)-1}^{(k)} \geq p_{\pi(3k)}$. Assume tat $R_{\pi(3k)-1}^{(k)} = p_{\pi(3k)}$; i.e.
% \begin{equation} \label{grnplus4}
% \pi(x) - \pi \left( \frac{x}{k} \right) \geq \pi(3k)-1
% \end{equation}
% for every $x \geq p_{\pi(3k)}$. Since $p_{\pi(3k)} \leq 3k$, we put $x=3k$ in \eqref{grnplus4} to obtain
% \begin{equation} \label{grnplus5}
% \pi(3k) - \pi \left( \frac{3k}{k} \right) \geq \pi(3k)-1
% \end{equation}
% On the other hand, we have $\pi(3k) - \pi(3k/k) = \pi(3k)-2 < \pi(3k)-1$, which contradicts \eqref{grnplus5}. So,
% \begin{equation} \label{grnplus6}
% R_{\pi(3k)-1}^{(k)} > p_{\pi(3k)}
% \end{equation}
% and, by induction, we obtain that $R_m^{(k)} = p_{m+1}$. Hence, $n \leq \pi(3k)-1$.
\end{proof}

\begin{rema}
Corollary \ref{k316} implies that for every $k \geq 29/3$ the prime numbers $p_{\pi(2k)}$ and $p_{\pi(3k)-1}$ are not $k$-Ramanujan primes.
\end{rema}

\begin{kor}
For each $m \in \N$ there exists $k = k(m) \geq 29/3$ such that $p_{\pi(2k)+1}, \ldots, p_{\pi(2k)+m}$ are all $k$-Ramanujan primes.
\end{kor}

\begin{proof}
By PNT, we obtain $\pi(3k)-2-\pi(2k) \to \infty$ as $k \to \infty$. Then use Corollary \ref{k316}(ii).
\end{proof}

\noindent
The next lemma provides an upper bound for $N_0(k)$.

\begin{lem} \label{l317}
Let $X_5 = X_5(k) = \max \{ X_0, kX_1, kY_0 \}$. Then the inequality
\begin{displaymath}
R_n^{(k)} \geq p_{\lceil kn/(k-1) \rceil}
\end{displaymath}
holds for every
\begin{displaymath}
n \geq \frac{k-1}{k}(\pi(X_5)+2).
\end{displaymath}
\end{lem}

\begin{proof}
We just set $t=-1$ in Theorem \ref{t301}.
\end{proof}

\noindent
The inequality in Theorem \ref{t313} becomes an equality in the following case.

\begin{thm} \label{t318}
If $k \geq 143.7$, then
\begin{displaymath}
N_0(k) = \pi(2k).
\end{displaymath}
\end{thm}

\begin{proof}
We set $A(x) = - 3.3/\log x$, $s=0$ and $Y_0 = 2$. Similarly to the proof of Corollary 3.6 from \cite{ca2}, we get that \eqref{309} is fulfilled for every
$x \geq Y_0$. By setting $B(x) = 1.17/\log x$ and $X_0=5.43$ and using Corollary 3.5 from \cite{ca2}, we see that the inequality \eqref{310} is fulfilled
for every $x \geq X_0$. Let
\begin{displaymath}
z(k) = \exp \left( \sqrt{3.3 + \frac{1}{4}\left( \log k - \frac{4.47}{\log k} \right)^2} - \frac{1}{2} \left( \log k - \frac{4.47}{\log k} \right) \right).
\end{displaymath}
It is easy to show that $x \geq z(k)$ is equivalent to the inequality \eqref{311}. By setting $X_6 = X_6(k) = \max\{ 2k, 5.43, kz(k) \}$ and using Lemma
\ref{l317}, we get that
\begin{displaymath}
N_0(k) \leq \left \lceil \frac{k-1}{k} ( \pi(X_6) + 2) \right \rceil.
\end{displaymath}
In the proof of Theorem \ref{t307}, we showed that $\widetilde r(k) \leq 3$ for every $k \geq 745.8$. Analogously, we get $z(k) \leq 2$ for every $k \geq
143.7$. Hence we obtain $X_6 = 2k$ and therefore $N_0(k) \leq \pi(2k)+2$ for every $k \geq 143.7$. By using a suitable upper bound for $\pi(x)$, it is easy
to see, that $\pi(2k) \leq k-1$ and $\pi(2k)+1 \leq k-1$ for every $k \geq 143.7$. Hence, by Proposition \ref{p315},
\begin{displaymath}
R_{\pi(2k)}^{(k)} = p_{\pi(2k)+1} = p_{\lceil k\pi(2k)/(k-1) \rceil}
\end{displaymath}
and
\begin{displaymath}
R_{\pi(2k)+1}^{(k)} \geq p_{\pi(2k)+2} = p_{\lceil k(\pi(2k)+1)/(k-1) \rceil}
\end{displaymath}
for every $k \geq 143.7$. So, we obtain $N_0(k) \leq \pi(2k)$ for every $k \geq 143.7$. Theorem \ref{t313} finishes the proof.
\end{proof}

\subsection{An upper bound for the $n$th $k$-Ramanujan prime}

After finding a lower bound for the $n$th $k$-Ramanujan prime, we find an upper bound by using the following two propositions, where $\Upsilon_{k}(x)$ is
defined by
\begin{equation} \label{upsilon}
\Upsilon_{k}(x) = \Upsilon_{k, a_1, \ldots, a_{m_1}, b_1, \ldots, b_{m_2}}(x) = \frac{x}{\log x - 1 - A(x)} \left( 1 - \frac{1}{k} - \frac{1}{k}\,
\frac{\log k - A(x) + B(x/k)}{\log (x/k) - 1 - B(x/k)} \right).
\end{equation}

\begin{prop} \label{p319}
If $x \geq \max\{ Y_0, k X_0\}$, then
\begin{displaymath}
\pi(x) - \pi \left( \frac{x}{k} \right) > \Upsilon_{k}(x).
\end{displaymath}
\end{prop}

\begin{proof}
We have
\begin{displaymath}
\pi(x) - \pi \left( \frac{x}{k} \right) > \frac{x}{\log x - 1 - A(x)} - \frac{x/k}{\log(x/k) - 1 - B(x/k)}
\end{displaymath}
for every $x \geq \max\{ Y_0, k X_0\}$ and see that the term on the right hand side is equal to $\Upsilon_{k}(x)$.
\end{proof}

\begin{prop} \label{p320}
For every sufficiently large $x$, the derivative $\Upsilon'_k(x) > 0$.
\end{prop}

\begin{proof}
We set $F(x) = F_{a_1, \ldots, a_{m_1}}(x) = x/(\log x - 1 - A(x))$ and
\begin{displaymath}
G(x) = G_{k,a_1, \ldots, a_{m_1}, b_1, \ldots, b_{m_2}}(x) = 1 - \frac{1}{k} - \frac{1}{k}\, \frac{\log k - A(x) + B(x/k)}{\log (x/k) - 1 - B(x/k)}.
\end{displaymath}
It is clear that $G(x) > 0$ for every sufficiently large $x$. Since $k > 1$, we obtain $B(x/k) > B(x) > A(x)$ as well as $\log(x/k) \leq \log x$ and
$\log(x/k) - 1 - B(x/k) > 0$ for every $x \geq \max \{Y_0, kX_0\}$. It follows that
\begin{equation} \label{330}
G'(x) > \frac{1}{k(\log(x/k) - 1 - B(x/k))}\, \left( \sum_{j=1}^{m_2} \frac{j \cdot b_j}{\log^{j+1} x} - \sum_{i=1}^{m_1} \frac{i \cdot a_i}{\log^{i+1}x}
\right)
\end{equation}
for every $x \geq \max \{Y_0, kX_0\}$. Since $b_1 > a_1$, we get
\begin{displaymath}
\sum_{j=1}^{m_2} \frac{j \cdot b_j}{\log^{j+1} x} - \sum_{i=1}^{m_1} \frac{i \cdot a_i}{\log^{i+1}x} \geq 0
\end{displaymath}
and therefore $G'(x) > 0$ for every sufficiently large $x$. We have $F(x) > 0$ for every $x \geq Y_0$. Further,
\begin{displaymath}
\log x - 2 - A(x) - \sum_{i=1}^{m_1} \frac{i \cdot a_i}{\log^{i+1}x} > 0
\end{displaymath}
and so
\begin{displaymath}
F'(x) = \frac{1}{(\log x - 1 - A(x))^2} \left( \log x - 2 - A(x) - \sum_{i=1}^{m_1} \frac{i \cdot a_i}{\log^{i+1}x} \right) > 0
\end{displaymath}
for every sufficiently large $x$. So, for every sufficiently large $x$, we get $\Upsilon'_k(x) = (F(x)G(x))' > 0$.
\end{proof}

\noindent
Now, let $m_1 = m_2 = 1$. By Proposition \ref{p320}, there exists an $X_7 = X_7(k, a_1, b_1)$ such that $\Upsilon'_k(x) > 0$ for every $x
\geq X_7$. Let $X_8 = X_8(b_1) \in \N$ be such that
\begin{equation} \label{X8}
p_n \geq n(\log p_n - 1 - b_1/\log p_n)
\end{equation}
for every $n \geq X_8$.

\begin{rema}
Clearly, $X_8 \leq \pi(X_0)+1$.
\end{rema}

\noindent
Let $\e_1 \geq 0$ and $\e_2 \geq 0$ be such that $\e_1 + \e_2 > 0$. We define
\begin{displaymath}
\e =
\begin{cases}
\e_1 &\text {\text{if} $\e_1 \neq 0$,}\\
\e_2 &\text {\text{otherwise}} \nonumber
\end{cases}
\end{displaymath}
and
\begin{displaymath}
\lambda = \frac{\e}{2} + \e_2 \cdot \text{sign}(\e_1) \left( 1 + \frac{\e}{2} \right).
\end{displaymath}
Let $S = S(k,a_1,b_1,X_0,\e_1,\e_2)$ be defined by
\begin{displaymath}
S = \exp \left( \sqrt{b_1 + \frac{2(1+\e)}{(k-1)\e} \left( b_1-a_1 + \frac{a_1\log k}{\log kX_0} \right) + \left( \frac{1}{2} + \frac{(1+\e)\log k}{(k-1)\e}
\right)^2} + \frac{1}{2} + \frac{(1+\e)\log k}{(k-1)\e} \right),
\end{displaymath}
and let $T = T(a_1,b_1,\e_1,\e_2)$ be defined by
\begin{displaymath}
T = \exp \left( \sqrt{b_1 + \frac{b_1-a_1}{\lambda} + \frac{a_1\log(1+\e_1)}{\lambda} + \left( \frac{1}{2} + \frac{\log(1+\e_1)}{2\lambda} \right)^2} +
\frac{1}{2} + \frac{\log(1+\e_1)}{2\lambda} \right).
\end{displaymath}
By defining $X_9 = X_9(k,a_1,b_1,Y_0,X_0,\e_1,\e_2,X_7)$ by
\begin{equation} \label{X9}
X_9 = \max \left \{ \frac{Y_0}{1+\e_1}, \frac{kX_0}{1+\e_1}, \frac{k\,S(k,a_1,b_1,\e_1,\e_2)}{1+\e_1}, T(a_1, b_1,\e_1,\e_2), \frac{X_7}{1+\e_1} \right \},
\end{equation}
we get, in view of \eqref{103}, the following result.

\begin{thm} \label{t321}
The inequality
\begin{displaymath}
R_n^{(k)} \leq (1+\e_1) p_{ \lceil (1+\e_2)kn/(k-1) \rceil}
\end{displaymath}
holds for every
\begin{equation} \label{333}
n \geq n_1 = \frac{k-1}{k(1+\e_2)} \max \{\pi(X_9)+1, X_8 \}.
\end{equation}
\end{thm}

\begin{proof}
For convenience, we write $t = t(n, k, \e_2) = \lceil (1+\e_2)nk/(k-1) \rceil$. Using Proposition \ref{p319}, we obtain
\begin{displaymath}
\pi(x) - \pi \left( \frac{x}{k} \right) > \Upsilon((1+\e_1) p_t)
\end{displaymath}
for every $x \geq (1+\e_1)p_t$. So to prove the claim, it is enough to show that
\begin{equation} \label{334}
\Upsilon((1+\e_1) p_t) \geq n.
\end{equation}
For this, we first show, using the definition \eqref{upsilon} of $\Upsilon((1+\e_1)p_t)$, that
\begin{equation} \label{335}
1 - \frac{1}{k} - \frac{1}{k}\, \frac{\log k - a_1/\log((1+\e_1)p_t) + b_1/\log((1+\e_1)p_t/k)}{\log((1+\e_1)p_t/k) - 1 - b_1/\log((1+\e_1)p_t/k)} >
\frac{k-1}{k} \left( 1 - \frac{\e}{2(1+\e)} \right).
\end{equation}
We have $(1+\e_1)p_t/k \geq S(k,a_1,b_1,X_0,\e_1,\e_2)$ and therefore
\begin{displaymath}
\frac{\e(k-1)}{2k(1+\e)} \left( \log((1+\e_1)p_t/k) - 1 - \frac{b_1}{\log((1+\e_1)p_t/k)} \right ) > \frac{\log k}{k} + \frac{b_1}{k \log((1+\e_1)p_t/k)}
- \frac{a_1}{k \log (1+\e_1)p_t}.
\end{displaymath}
From this inequality, we obtain \eqref{335}. So for the proof of \eqref{334}, using the definition \eqref{upsilon} of $\Upsilon((1+\e_1)p_t)$, it suffices
to show that the inequality
\begin{equation} \label{336}
\frac{k-1}{k} \left( 1 - \frac{\e}{2(1+\e)} \right) \cdot \frac{(1 + \e_1) p_t}{\log (1+ \e_1) p_t - 1 - a_1/\log((1+\e_1)p_t)} \geq n
\end{equation}
is fulfilled. Since $p_t \geq T(a_1,b_1,\e_1,\e_2)$, we get
\begin{equation} \label{337}
\frac{k-1}{k}\left( 1 - \frac{\e}{2(1+\e)} \right) \cdot \frac{(1+\e_1)\, t\, (\log p_t - 1 - b_1/\log(p_t))}{\log((1+\e_1)p_t)  - 1 -
a_1/\log((1+\e_1)p_t)} \geq n.
\end{equation}
Since $t \geq X_8$, we have $p_t > t(\log p_t - 1 - b_1/\log(p_t))$. Using \eqref{337}, we get \eqref{336} and therefore \eqref{334}.
\end{proof}

\noindent
Now, let $m_1 = m_2 = 1,a_1 = 1$ and $b_1 = 1.17$. In the next lemma, we determine an explicit $X_{10} = X_{10}(k)$ such that $\Upsilon'_k(x) > 0$ for every
$x \geq X_{10}$.

\begin{lem} \label{l322}
Let $X_{10} = X_{10}(k) = \max \{ k X_{11}, e^{2.547}, 5.43k \}$, where
\begin{displaymath}
X_{11} = X_{11}(k) = \exp \left( \sqrt{1.17 + \frac{1.17}{k-1} + \left( \frac{1}{2} + \frac{\log k}{2(k-1)}\right)^2} + \frac{1}{2} + \frac{\log k}{2(k-1)}
\right).
\end{displaymath}
Then, $\Upsilon'_k(x) > 0$ for every $x \geq X_{10}$.
\end{lem}

\begin{proof}
We set $F(x) = x/(\log x - 1 - 1/\log x)$ and
\begin{displaymath}
G(x) = \frac{k-1}{k} - \frac{1}{k} \, \frac{\log k - 1/\log x + 1.17/\log(x/k)}{\log(x/k) - 1 - 1.17/\log(x/k)}.
\end{displaymath}
We have $F'(x) \geq 0$ for every $x \geq e^{2.547}$ and $F(e^{2.547}) > 0$. Hence, $F(x) > 0$ for every $x \geq e^{2.547}$. For every $x \geq 5.43k$ we have
$\log(x/k)-1-1.17/\log(x/k) > 0$ and, using \eqref{330}, we obtain $G'(x) > 0$. It is easy to see that $x \geq kX_{11}$ implies $G(x) > 0$ and we get
$\Upsilon'_k(x) = (F(x)G(x))' > 0$ for every $x \geq X_{10}$.
\end{proof}

\noindent
For Ramanujan primes, we obtain the following result.

\begin{prop} \label{p323}
If $t > 48/19$, then for every $n\in \N$
\begin{displaymath}
R_{n} \leq p_{\lceil tn \rceil}.
\end{displaymath}
\end{prop}

\begin{proof}
By \cite{ca2}, we choose $Y_0 = 468049$ and $X_0 = 5.43$. Then we have $p_n > n(\log p_n - 1 - 1.17/\log p_n)$ for every $n \geq 4$. Since this inequality
is also true for every $1 \leq n \leq 3$, we choose $X_8=1$ in Theorem \ref{t321}. By Lemma \ref{l322}, we choose $X_7(k,1,1.17) = X_{10}$ in Theorem
\ref{t321}. It follows
\begin{displaymath}
R_n^{(k)} \leq (1+\e_1) p_{ \lceil (1+\e_2)kn/(k-1) \rceil}
\end{displaymath}
for every
\begin{displaymath}
n \geq \frac{k-1}{k(1+\e_2)}(\pi(X_{12})+1),
\end{displaymath}
where $X_{12} = X_{12}(k,\e_1,\e_2) = X_9(k,1,1.17,468049,5.43,\e_1,\e_2,X_{10})$. Let $s = 48/19$ and $t > s$. We set $k=2$, $\e_1=0$ and $\e_2 = 5/19$ in
Theorem \ref{t321}, and we get $R_n \leq p_{\lceil sn \rceil}$ for every $n \geq 19536$. By using a computer, we check that $R_n \leq p_{\lceil sn \rceil} $
for every $20 \leq n \leq 19535$ and for every $1 \leq n \leq 18$. For $n= 19$, we have $p_{\lceil 19s \rceil} < R_{19} = p_{49} \leq p_{\lceil 19t
\rceil}$.
\end{proof}

\begin{rema}
In 2014, Srinivasan \cite{sri} proved independently for the case $k=2$ that for every $\e > 0$ there exists an integer $N$ such that $R_n < p_{ \lfloor
2n(1+\e) \rfloor}$ for every $n > N$. Further, she showed that $R_n \leq p_{ \lfloor 2.6n \rfloor}$ holds for every $n\in \N$. Using Proposition \ref{p323},
we obtain an improvement of the last inequality, namely that $R_n \leq p_{\lceil 2.53n \rceil}$ holds for every $n \in \N$.
\end{rema}

\section{On the difference $R_n^{(k)} - p_{\lceil nk/(k-1) \rceil}$}

Another question that arises in view of \eqref{103}, is the size of
\begin{equation} \label{438}
R_n^{(k)} - p_{\lceil nk/(k-1) \rceil}.
\end{equation}
In Proposition \ref{p305}, we yield a lower bound for \eqref{438}, which improves the lower bound in \eqref{106}. The goal in this section is to improve the
upper bound in \eqref{106}. In order to do this, we set $A(x) = 0$ and $B(x) = b_1/\log x$. By \cite{pd}, we choose $Y_0 = 5393$. Let $\e_1$, $\e_2$,
$\delta_1$ and $\delta_2$ all be positive and let
\begin{displaymath}
\eta(k) = k \left( \sqrt{b_1\left(1+\frac{1}{\delta_1}\right) + \left( \frac{1}{2}+\frac{\log k}{2\delta_1}\right)^2} + \frac{1}{2}+\frac{\log k}{2\delta_1}
\right).
\end{displaymath}
In addition, we set
\begin{displaymath}
X_{13} = X_{13}(k,b_1,X_0,\delta_1,\delta_2) = \max \{ 7477, kX_0, \eta(k), ke^{b_1/\delta_2}\}.
\end{displaymath}
As in the proof of Lemma \ref{l322}, we get that $\Upsilon'_k(x) \geq 0$ for every $x \geq X_{14}$, where
\begin{equation} \label{X14}
X_{14} = X_{14}(k) = \max\{5393, kX_0, kX_{11} \},
\end{equation}
where $X_{11} = X_{11}(k)$ is given by Lemma \ref{l322}. Further, let
\begin{displaymath}
X_{15} = X_{15}(k,b_1,\e_2) = X_9(k,0,b_1,5393,X_0,0,\e_2,X_{14}),
\end{displaymath}
where $X_9(k,0,b_1,5393,X_0,0,\e_2,X_{14})$ is given by \eqref{X9}, as well as
\begin{equation} \label{X16}
X_{16} = X_{16}(k,b_1,X_0,\e_2,\delta_1,\delta_2) = \frac{k-1}{k} \max \left \{ \pi(X_5) + 2, \frac{X_8}{1+\e_2}, \pi(X_{13})+1,
\frac{\pi(X_{15})+1}{1+\e_2} \right \},
\end{equation}
where $X_5$ is given by Lemma \ref{l317} and $X_8$ is defined by \eqref{X8}.

\subsection{On the difference $kn \log R_n^{(k)}/(k-1) - R_n^{(k)}$}

We consider the difference
\begin{displaymath}
\frac{kn}{k-1} \log R_n^{(k)} - R_n^{(k)}.
\end{displaymath}
The results below for this difference will be useful to find an upper bound of the difference \eqref{438}.

\begin{prop} \label{p401}
If $n \geq X_{16}$, where $X_{16}$ is defined by \eqref{X16}, then
\begin{displaymath}
\frac{kn}{k-1} \log R_n^{(k)} - R_n^{(k)} > \left( 1 - \frac{(1+\e_2)(1+\delta_1)(\log k + \delta_2)}{k-1} \right) \frac{kn}{k-1}.
\end{displaymath}
\end{prop}

\begin{proof}
Dusart \cite{pd} proved that
\begin{equation} \label{441}
\pi(x) > \frac{x}{\log x - 1}
\end{equation}
holds for every $x \geq 5393$. Using this estimate, we get
\begin{equation} \label{442}
\pi(x) - \pi \left( \frac{x}{k} \right) >  \frac{(k-1)x}{k(\log x - 1)} - \frac{x}{k(\log x - 1)} \cdot \frac{\log k + b_1/\log(x/k)}{\log x/k - 1 -
b_1/\log(x/k)}
\end{equation}
for every $x \geq \max \{ 5393, kX_0\}$. Since $\log x - 1 \leq (1+\delta_1)(\log x/k - 1 - b_1/\log(x/k))$ for every $x \geq \eta(k)$, we can use
\eqref{442} and the inequality $b_1/\log(x/k) \leq \delta_2$, which is fulfilled for every $x \geq ke^{b_1/\delta_2}$, to see that
\begin{equation} \label{443}
\pi(x) - \pi \left( \frac{x}{k} \right) > \frac{(k-1)x}{k(\log x - 1)} - \frac{(1+\delta_1)(\log k + \delta_2)}{k} \cdot \frac{x}{(\log x - 1)^2}
\end{equation}
for every $x \geq X_{13}$. Since $n \geq X_{16} \geq (k-1)\max \{ \pi(X_5) + 2, \pi(X_{13})+1 \}/k$, we have $R_n^{(k)} \geq p_{ \lceil kn/(k-1) \rceil}
\geq X_{13}$. So we choose $x = R_n^{(k)}$ in \eqref{443} and obtain
\begin{displaymath}
\pi(R_n^{(k)}) - \pi \left( \frac{R_n^{(k)}}{k} \right) > \frac{(k-1)R_n^{(k)}}{k(\log x - 1)} - \frac{(1+\delta_1)(\log k + \delta_2)}{k} \cdot
\frac{R_n^{(k)}}{(\log R_n^{(k)} - 1)^2}.
\end{displaymath}
Using Proposition \ref{p207}, we get the inequality
\begin{equation} \label{444}
\frac{kn}{k-1} \log R_n^{(k)} - R_n^{(k)} > \frac{kn}{k-1} - \frac{(1+\delta_1)(\log k + \delta_2)}{k-1} \cdot \frac{R_n^{(k)}}{\log R_n^{(k)} - 1}.
\end{equation}
From \eqref{315}, it follows that
\begin{equation} \label{445}
\pi(x) > \frac{x}{\log x - 1} + 1
\end{equation}
for every $x \geq 470077$. We check with a computer that \eqref{445} also holds for every $7477 \leq x \leq 470077$. Using \eqref{444} and \eqref{445}, we
get
\begin{equation} \label{446}
\frac{kn}{k-1} \log R_n^{(k)} - R_n^{(k)} > \frac{kn}{k-1} - \frac{(1+\delta_1)(\log k + \delta_2)}{k-1} (\pi(R_n^{(k)})-1).
\end{equation}
Since $n \geq (k-1)(\pi(X_{15})+1)/(k(1+\e_2))$, we have $R_n^{(k)} \leq p_{ \lceil (1+\e_2)kn/(k-1) \rceil}$; i.e. $\pi(R_n^{(k)}) - 1 \leq
(1+\e_2)nk/(k-1)$. Now use \eqref{446}.
\end{proof}

\begin{kor} \label{k402}
Let $\e_2$, $\delta_1$ and $\delta_2$ all be positive so that
\begin{displaymath}
(1+\e_2)(1+\delta_1)(\log k + \delta_2) < k-1.
\end{displaymath}
If $n \geq X_{16}$, where $X_{16}$ is defined by \eqref{X16}, then
\begin{displaymath}
\frac{kn}{k-1} \log R_n^{(k)} >  R_n^{(k)}.
\end{displaymath}
\end{kor}

\begin{proof}
Follows directly from Proposition \ref{p401}.
\end{proof}

\begin{rema}
Nicholson \cite{nic}, \cite{nic2} proved that $R_n \geq 2n \log R_n$ is fulfilled for every $n \in M = \{ 1, 2, 3, 4, 5, 6, 7, 10,$ $13, 14, 15, 16, 17,
19, 20, 21, 29, 31, 33, 34, 43, 44, 46, 68, 97, 98, 145, 166, 167, 168, 201 \}$ and for every $n \in \N\setminus M$ we have $R_n < 2n \log R_n$. Corollary
\ref{k402} generalizes the last inequality to $k$-Ramanujan primes.
\end{rema}

\begin{rema}
Amersi, Beckwith, Miller, Ronan and Sondow \cite{abmrs} showed that there exists a positive constant $c = c(k)$ such that
\begin{equation} \label{447}
\left| \frac{kn}{k-1} \log R_n^{(k)} - R_n^{(k)} \right| \leq \frac{c R_n^{(k)}}{\log R_n^{(k)}}
\end{equation}
for every sufficiently large $n$. With Corollary \ref{k402}, we obtain an improvement of the lower bound in \eqref{447}.
\end{rema}

\noindent
We end this section by finding an upper bound for $kn\log(R_n^{(k)})/(k-1) - R_n^{(k)}$.

\begin{prop} \label{p403}
Let $\e > 0$ and
\begin{displaymath}
X_{17} = X_{17}(k,b_1,\e) = \max \{ X_0, 5393k, e^{b_1/\e}, e^{b_1/\log k}, X_2(k,1,1,1,0,b_1) \},
\end{displaymath}
where $X_2$ is defined by \eqref{X2}. If
\begin{displaymath}
n \geq \frac{k-1}{k} (\pi(X_{17})+1),
\end{displaymath}
then
\begin{displaymath}
\frac{kn}{k-1} \log R_n^{(k)} - R_n^{(k)} < \left( 1 - \frac{\log k - \e k}{k-1} \right) \frac{kn}{k-1}.
\end{displaymath}
\end{prop}

\begin{proof}
Using \eqref{441}, we obtain
\begin{displaymath}
\pi(x) - \pi \left( \frac{x}{k} \right) <  \frac{(k-1)x}{k(\log x - 1 - b_1/\log x)} - \frac{x}{k(\log x - 1 - b_1/\log x)} \cdot \frac{\log k -
b_1/\log x}{\log x/k - 1}
\end{displaymath}
for every $x \geq \max \{ 5393k, X_0\}$. Since $\log x - 1 - b_1/\log x \geq \log(x/k) - 1$ for every $x \geq e^{b_1/\log k}$, we get
\begin{equation} \label{448}
\pi(x) - \pi \left( \frac{x}{k} \right) < \frac{(k-1)x}{k(\log x - 1 - b_1/\log x)} - \frac{x(\log k - b_1/\log x)}{k(\log x - 1 - b_1/\log x)^2}
\end{equation}
for every $x \geq X_{17}$. Since $n \geq (k-1)(\pi(X_{17})+1)/k$, we have $R_n^{(k)} > p_{ \lceil kn/(k-1) \rceil } \geq X_{17}$. So, we set $x = R_n^{(k)}$
in \eqref{448} and, using Proposition \ref{p207}, get
\begin{displaymath}
R_n^{(k)} - \frac{kn}{k-1} \log R_n^{(k)} > \frac{\log k - b_1/ \log R_n^{(k)}}{k-1} \cdot \frac{R_n^{(k)}}{\log R_n^{(k)} - 1 - b_1/\log R_n^{(k)}} -
\frac{kn}{k-1} - \frac{kn}{k-1} \cdot \frac{b_1}{\log R_n^{(k)}}.
\end{displaymath}
Since $R_n^{(k)} > p_{\lceil nk/(k-1) \rceil} \geq \max \{ X_0, e^{b_1/\log k} \}$, we obtain
\begin{displaymath}
R_n^{(k)} - \frac{kn}{k-1} \log R_n^{(k)} > \frac{\log k - b_1/ \log R_n^{(k)}}{k-1} \cdot \pi(R_n^{(k)}) - \frac{kn}{k-1} - \frac{kn}{k-1} \cdot
\frac{b_1}{\log R_n^{(k)}}.
\end{displaymath}
We have $\pi(R_n^{(k)}) \geq \pi(p_{\lceil nk/(k-1) \rceil}) \geq nk/(k-1)$. Therefore,
\begin{displaymath}
R_n^{(k)} - \frac{kn}{k-1} \log R_n^{(k)} > \frac{\log k - b_1/ \log R_n^{(k)}}{k-1} \cdot \frac{kn}{k-1} - \frac{kn}{k-1} - \frac{kn}{k-1} \cdot
\frac{b_1}{\log R_n^{(k)}}.
\end{displaymath}
It remains to notice that $b_1/\log R_n^{(k)} < \e$.
\end{proof}

\subsection{An upper bound for $R_n^{(k)} - p_{\lceil nk/(k-1) \rceil}$}

Now, we find an upper bound for \eqref{438}, which improves the upper bound in \eqref{106}. We define
\begin{displaymath}
X_{18} = X_{18}(k,b_1,\e_1) = X_9(k,0,b_1,5393,X_0,\e_1,0,X_{14}),
\end{displaymath}
where $X_9$ is defined by \eqref{X9} and $X_{14}$ is given by \eqref{X14}. Let $\e_3 > 0$ and let $X_{19} = X_{19}(\e_3)$ be such that
\begin{equation} \label{X19}
\log \log x < \e_3 \log x
\end{equation}
for every $x \geq X_{19}$. By setting
\begin{equation} \label{X20}
X_{20} = X_{20}(k,b_1,X_0,\e_1,\e_2,\e_3,\delta_1,\delta_2) = \max \left \{ \frac{k-1}{k} X_8, X_{16}, \frac{k-1}{k}(\pi(X_{18})+1), \frac{k-1}{k}
X_{19} \right \},
\end{equation}
where $X_8, X_{16}= X_{16}(k,b_1,X_0,\e_2,\delta_1,\delta_2)$ are given by \eqref{X8}, \eqref{X16}, respectively, and
\begin{equation} \label{gamma}
\gamma = \gamma(k,\e_1,\e_2,\e_3,\delta_1, \delta_2) = \left(\frac{(1+\e_2)(1+\delta_1)(\log k + \delta_2)}{k-1} + \log( (1+\e_1)(1+\e_3)) \right)
\frac{k}{k-1},
\end{equation}
we obtain the following result.

\begin{thm} \label{t404}
If $n \geq n_2 = X_{20}$, then
\begin{displaymath}
R_n^{(k)} - p_{\lceil nk/(k-1) \rceil} < \gamma n.
\end{displaymath}
\end{thm}

\begin{proof}
By Theorem \ref{t321}, we get the inequality
\begin{equation} \label{452}
R_n^{(k)} \leq (1+\e_1)p_{\lceil nk/(k-1) \rceil},
\end{equation}
since $n \geq (k-1) \max \{\pi(X_{18})+1,X_8 \}$. By \cite{rs}, we have
\begin{equation} \label{453}
p_n \leq n(\log n + \log \log n)
\end{equation}
for every $n \geq 6$. Hence, using \eqref{452},
\begin{displaymath}
\frac{kn}{k-1} \log R_n^{(k)} \leq \frac{kn}{k-1} \log(1+\e_1) + \frac{kn}{k-1} \log \left\lceil \frac{kn}{k-1} \right\rceil + \frac{kn}{k-1} \log \left(
\log \left\lceil \frac{kn}{k-1} \right\rceil + \log \log \left\lceil \frac{kn}{k-1} \right\rceil \right).
\end{displaymath}
Since $nk/(k-1)\geq X_{19}$, we get $\log \log \lceil nk/(k-1) \rceil \leq \e_3 \log \lceil nk/(k-1) \rceil$ and therefore
\begin{displaymath}
\frac{kn}{k-1} \log R_n^{(k)} \leq \frac{kn}{k-1} \log((1+\e_1)(1+\e_3)) + \frac{kn}{k-1} \log \left\lceil \frac{kn}{k-1} \right\rceil + \frac{kn}{k-1}
\log \log \left\lceil \frac{kn}{k-1} \right\rceil.
\end{displaymath}
Using the estimate for $p_n$ proved by Dusart \cite{pd3}, we get
\begin{equation} \label{454}
\frac{kn}{k-1} \log R_n^{(k)} - p_{\lceil nk/(k-1) \rceil} \leq \frac{kn}{k-1} \left( \log ((1+\e_1)(1+\e_3)) + 1 \right).
\end{equation}
By Proposition \ref{p401}, we obtain
\begin{displaymath}
R_n^{(k)} - p_{\lceil nk/(k-1) \rceil} < \left( \frac{(1+\e_2)(1+\delta_1)(\log k + \delta_2)}{k-1} - 1 \right) \frac{kn}{k-1} + \frac{kn}{k-1} \log
R_n^{(k)} - p_{\lceil nk/(k-1) \rceil}.
\end{displaymath}
Now use \eqref{454}.
\end{proof}

\begin{kor} \label{k404}
The sequence $((R_n^{(k)} - p_{\lceil nk/(k-1) \rceil})/n)_{n \in \N}$ is bounded.
\end{kor}

\begin{proof}
Follows from Proposition \ref{p305} and Theorem \ref{t404}.
\end{proof}

\begin{rema}
In particular, Corollary \ref{k404} gives a positive answer to the question raised by Sondow \cite{so2} in 2013, whether the sequence $((R_n -
p_{2n})/n)_{n \in \N}$ is bounded.
\end{rema}

\section{On the number of $k$-Ramanujan primes $\leq x$}

Let $\pi_k(x)$ be the number of $k$-Ramanujan primes less than or equal to $x$. Amersi, Beckwith, Miller, Ronan and Sondow \cite{abmrs} proved that
\begin{displaymath}
\frac{\pi_{k}(x)}{\pi(x)} \sim \frac{k-1}{k} \q\q (x \to \infty)
\end{displaymath}
by showing that there exists a positive constant $\beta_2 = \beta_2(k)$ such that for every sufficiently large $n$,
\begin{equation} \label{555}
|\rho_k(n)| \leq \frac{\beta_2 \log \log n}{\log n},
\end{equation}
where
\begin{equation} \label{rhokx}
\rho_k(x) = \frac{k-1}{k} - \frac{\pi_k(x)}{\pi(x)}.
\end{equation}
Now we improve the lower bound in \eqref{555}.

\begin{prop} \label{p501}
If $x \geq R_{N(k)}^{(k)}$, where $N(k)$ is defined by \eqref{Nk05}, then
\begin{displaymath}
\rho_k(x) > 0.
\end{displaymath}
\end{prop}

\begin{proof}
Let $n \geq N(k)$ be such that $R_n^{(k)} \leq x < R_{n+1}^{(k)}$. Hence, $R_n^{(k)} > p_{\lceil nk/(k-1) \rceil}$, and we get $\pi(x) > nk/(k-1)$. Since
$\pi_k(x) = n$, our proposition is proved.
\end{proof}

\noindent
With the same method as in \cite{abmrs}, we improve the upper bound of \eqref{555} by using the following lemma.

\begin{lem} \label{l502}
If $s \geq 0$ and
\begin{equation} \label{c0s}
c_0(s) = \max \{ 4, 4s, (\pi(2+s)-1)\log 2 \},
\end{equation}
then for every $n \in \N$,
\begin{displaymath}
\pi(p_n + s n) - \pi(p_n) \leq \frac{c_0 n}{ \log p_n}.
\end{displaymath}
\end{lem}

\begin{proof}
The claim obviously holds for $s=0$. So let $s> 0$. If $n = 1$, we get
\begin{displaymath}
\pi(p_n + sn) - \pi(p_n) = \pi(2+s) - 1 \leq \frac{c}{\log 2}.
\end{displaymath}
Let $n \geq 2$. If $n < 3/s$, we obtain, by using the inequality $p_n \leq n^2$, which holds for every $n \geq 2$,
\begin{displaymath}
\pi(p_n + sn) - \pi(p_n) \leq 1 \leq \frac{4n}{2\log n} \leq \frac{cn}{\log p_n}.
\end{displaymath}
So let $n \geq \max\{2, 3/s\}$. Montgomery and Vaughan \cite{mv} proved that
\begin{displaymath}
\pi(M+N) - \pi(M) \leq \frac{2N}{\log N}
\end{displaymath}
for every $M,N \in \N$ with $N \geq 2$. By setting $M = p_n$ and $N = \lfloor sn \rfloor$, we get
\begin{equation} \label{558}
\pi(p_n + sn) - \pi(p_n) = \pi(p_n + \lfloor sn \rfloor) - \pi(p_n) \leq \frac{2\lfloor sn \rfloor}{\log \lfloor sn \rfloor} \leq \frac{2sn}{\log sn}.
\end{equation}
The last inequality in \eqref{558} holds, since $sn \geq 3$. If $s\geq 1$, we get, by using \eqref{558},
\begin{equation} \label{559}
\pi(p_n + sn) - \pi(p_n) \leq \frac{2sn}{\log n}.
\end{equation}
If $0 < s < 1$, then, using \eqref{558},
\begin{displaymath}
\pi(p_n + sn) - \pi(p_n) \leq \frac{2n}{\log n},
\end{displaymath}
since $3 \leq sn < n$. Combined with \eqref{559}, we obtain the inequality
\begin{displaymath}
\pi(p_n + sn) - \pi(p_n) < \frac{2\max \{ 1,s\}n}{\log n}
\end{displaymath}
for every $s > 0$. Again, using the inequality $p_n \leq n^2$, we get
\begin{displaymath}
\pi(p_n + sn) - \pi(p_n) < \frac{4\max\{1,s\}n}{\log p_n}.
\end{displaymath}
So the lemma is proved.
\end{proof}

\noindent
Now we obtain the following result, which leads to an improvement of the upper bound in \eqref{555}. Let $\e_4 > 0$ and $X_{21} = X_{21}(\e_3,\e_4)$ be such
that
\begin{displaymath}
\log(1+\e_3) + \log(x+1) + \log(x + \log(x+1)) \leq \e_4 x
\end{displaymath}
for every $x \geq X_{21}$. In addition, we define
\begin{displaymath}
c_1 = c_1(k,\e_1,\e_2,\e_3,\e_4,\delta_1,\delta_2) = 1+\e_4 + c_0(\gamma),
\end{displaymath}
where $\gamma$ and $c_0(\gamma)$ are defined by \eqref{gamma} and \eqref{c0s}, respectively, and take $c_2 \in \R$ with $c_2 > c_1$. Further, define
\begin{displaymath}
X_{22} = X_{22}(k,b_1,X_0,\e_1,\e_2,\e_3,\e_4,\delta_1,\delta_2) = \max \left \{ X_{20}, X_{21}, \log \frac{k}{k-1} \right \},
\end{displaymath}
where $X_{20}$ is defined by \eqref{X20}, and
\begin{displaymath}
X_{23} = X_{23}(k,b_1,X_0,\e_1,\e_2,\e_3,\e_4,\delta_1,\delta_2,c_2) = \max \left \{ X_{19}, \left \lceil \frac{kX_{22}}{k-1} \right \rceil + 1, \pi \left(
2^{c_1/(c_2-c_1)} \right) + 1 \right \},
\end{displaymath}
where $X_{19}$ is defined by \eqref{X19}.

\begin{thm} \label{t503}
If $x \geq n_3 = p_{X_{23}}$, then
\begin{displaymath}
\rho_k(x) \leq \frac{c_2}{\log x}.
\end{displaymath}
\end{thm}

\begin{proof}
First, we prove the claim for $x=p_n$ with $n \geq X_{23}$.
% Since $n \geq X_{26} > X_{25} \geq X_{22}$, we have, by Theorem \ref{t404},
% \begin{equation} \label{541}
% R_n^{(k)} < p_{\lceil nk/(k-1) \rceil} + \gamma n.
% \end{equation}
Let $m \in \N$ be such that $\lceil mk/(k-1)\rceil = n$. Then
\begin{displaymath}
\left \lceil \frac{km}{k-1} \right \rceil = n \geq X_{23} \geq \left \lceil \frac{kX_{22}}{k-1} \right \rceil + 1
\end{displaymath}
and it follows $m \geq X_{22} \geq X_{20}$. Hence by Theorem \ref{t404}, we have
\begin{displaymath}
R_m^{(k)} < p_{\lceil mk/(k-1) \rceil} + \gamma m = p_n + \gamma m.
\end{displaymath}
Since $m \leq n$ and $\gamma > 0$, we get $R_m^{(k)} < p_n + \gamma n$. It follows
\begin{displaymath}
\pi_k(p_n) \geq m - (\pi_k(p_n + \gamma n) - \pi_k(p_n)).
\end{displaymath}
Since every $k$-Ramanujan prime is prime, we obtain
\begin{displaymath}
\pi_k(p_n) \geq m - (\pi(p_n + \gamma n) - \pi(p_n))
\end{displaymath}
and using Lemma \ref{l502}, we get the inequality
\begin{equation} \label{560}
\frac{\pi_k(p_n)}{\pi(p_n)} = \frac{\pi_k(p_n)}{n} \geq \frac{m}{n} - \frac{c_0(\gamma)}{\log p_n}.
\end{equation}
Since
\begin{displaymath}
\frac{m}{n} = \frac{m}{\lceil\frac{mk}{k-1}\rceil} \geq \frac{m}{\frac{mk}{k-1} + 1} = \frac{k-1}{k} - \frac{(k-1)^2}{mk^2 + k(k-1)} \geq \frac{k-1}{k} -
\frac{1}{m},
\end{displaymath}
we get, using \eqref{560} and the definition \eqref{rhokx} of $\rho_k(x)$,
\begin{equation} \label{561}
\rho_k(p_n) \leq \frac{1}{m} + \frac{c_0(\gamma)}{\log p_n}.
\end{equation}
Using \eqref{453} as well as $X_{19} \leq n \leq (m+1)k/(k-1)$, we obtain
\begin{align*}
\log p_n & \leq \log n + \log( \log n + \log \log n) \\
& \leq \log (m+1) + \log \frac{k}{k-1} + \log \left( \log n + \e_3 \log n \right) \\
& \leq \log (m+1) + \log \frac{k}{k-1} + \log (1+\e_3) + \log \left( \log(m+1) + \log \frac{k}{k-1} \right).
\end{align*}
Since $m \geq X_{22} \geq \log(k/(k-1))$, it follows that
\begin{displaymath}
\log p_n \leq m + \log(1+\e_3) + \log(m+1) + \log(m + \log(m+1)).
\end{displaymath}
Using $m \geq X_{22} \geq X_{21}$, we get $\log p_n \leq (1+\e_4)m$ and using \eqref{561}, we obtain the inequality
\begin{equation} \label{562}
\rho_k(p_n) \leq \frac{1+\e_4}{\log p_n} + \frac{c_0(\gamma)}{\log p_n} = \frac{c_1}{\log p_n}.
\end{equation}
So the theorem is proved in the case $x= p_n$.

Now let $x \in \R$ with $x \geq p_{X_{23}}$ and let $n \geq X_{23}$ be such that $p_n \leq x < p_{n+1}$. Using \eqref{562}, we get
\begin{equation} \label{563}
\rho_k(x) = \rho_k(p_n) \leq \frac{c_1}{\log p_n}.
\end{equation}
Since $n \geq \pi(2^{c_1/(c_2-c_1)}) + 1$, we get $p_n \geq 2^{c_1/(c_2-c_1)}$, which is equivalent to
\begin{displaymath}
\frac{c_1}{\log p_n} \leq \frac{c_2}{\log 2p_n}.
\end{displaymath}
Using \eqref{563}, it follows that
\begin{displaymath}
\rho_k(x) \leq \frac{c_2}{\log 2p_n}.
\end{displaymath}
From Bertrand's postulate, it follows that $2p_n \geq p_{n+1}$ for every $n\in\N$. So,
\begin{displaymath}
\rho_k(x) \leq \frac{c_2}{\log 2p_n} \leq \frac{c_2}{\log p_{n+1}} \leq \frac{c_2}{\log x}
\end{displaymath}
and the theorem is proved in general.
\end{proof}

\section{On a conjecture of Mitra, Paul and Sarkar}

In 2009, Mitra, Paul and Sarkar \cite{mps} made the following

\begin{conj} \label{conj601}
If $m,n \in \N$ with $n \geq \lceil 1.1\log 2.5m \rceil$, then
\begin{displaymath}
\pi(mn) - \pi(n) \geq m-1.
\end{displaymath}
\end{conj}

\noindent
We prove this conjecture for every sufficiently large $m$ by using the following proposition.

\begin{prop} \label{p602}
Let $\e > 0$. Then
\begin{displaymath}
\pi(mx) - \pi(x) \geq m-1
\end{displaymath}
for every sufficiently large $m$ and every $x \geq (1+\e)p_m/m$.
\end{prop}

\begin{proof}
We set  $A(x) = 0$, $B(x) = b_1/\log x$ and choose $Y_0 = 5393$ by \cite{pd}. By Theorem \ref{t321}, we get
\begin{equation} \label{664}
R_n^{(k)} \leq (1+\e) p_{\lceil kn/(k-1) \rceil}
\end{equation}
for every $n \in \N$ with $n \geq (k-1)\max\{\pi(X_{24})+ 1, X_8\}/k$, where $X_8$ is defined by \eqref{X8} and
\begin{displaymath}
X_{24} = X_{24}(k,b_1,X_0,\e) = X_9(k,0,b_1,5393,X_0,\e,0,X_{10}),
\end{displaymath}
where $X_9$ is given by \eqref{X9} and $X_{10}$ is defined by Lemma \ref{l322}. Now let $m \in \N$ be sufficiently large, so that
\begin{displaymath}
m \geq \max\{\pi(X_{24}(m,b_1,X_0,\e))+ 1, X_8\}.
\end{displaymath}
Then $m-1 \geq (m-1)\max\{\pi(X_{24}(m,b_1,X_0,\e))+ 1, X_8\}/m$ and, by \eqref{664}, we get the inequality $R_{m-1}^{(m)} \leq (1+\e) p_m$.
\end{proof}

\begin{kor} \label{k603}
The conjecture of Mitra, Paul and Sarkar holds for every sufficiently large $m$.
\end{kor}

\begin{proof}
Let $0 < \e < 0.1$. By \cite{rs}, we have $p_m \leq m(\log m + \log \log m -0.5)$ for every $m\geq 20$. So, we get
\begin{displaymath}
(1+\e)p_m/m \leq (1+\e) (\log m + \log \log m - 0.5) \leq \lceil 1.1\log 2.5m \rceil
\end{displaymath}
for every sufficiently large $m$. It remains to apply Proposition \ref{p602}.
\end{proof}

\section{Acknowledgement}
I would like to thank Benjamin Klopsch for the helpful conversations. Also I would like to thank Elena Klimenko and Anitha Thillaisundaram for their careful
reading of the paper.

\vspace{5mm}

\textsc{Mathematisches Institut, Heinrich-Heine-Universität Düsseldorf, 40225 Düsseldorf}, \textsc{Germany}

\emph{E-mail address}: \texttt{axler@math.uni-duesseldorf.de}

\end{document}